\documentclass[11pt]{amsart}
\usepackage{graphicx}
\usepackage{amsmath, amssymb}
\usepackage{verbatim}
\usepackage{color}
\newtheorem{thm}{Theorem}[section]
\newtheorem{cor}[thm]{Corollary}
\newtheorem{conj.}[thm]{Conjecture}
\newtheorem{lem}[thm]{Lemma}
\newtheorem{prop}[thm]{Proposition}
\theoremstyle{definition}
\newtheorem{defn}[thm]{Definition}
\theoremstyle{remark}
\newtheorem{rem}[thm]{Remark}
\numberwithin{equation}{section}
\newtheorem{exa}[thm]{Example}

\newcommand{\kf}{$K$-frames}
\newcommand{\hs}{Hilbert spaces}
\def\CC{\mathbb{C}}
\begin{document}

\title[Controlled $K$-frames in Hilbert Spaces]
{Controlled $K$-frames in Hilbert Spaces}%
\author[A. Rahimi$^1$, Sh. Najafzadeh$^{2}$ and M. Nouri$^{3}$]{A. Rahimi$^1$, Sh. Najafzadeh$^{2}$ and M. Nouri$^{3}$ }
\address{$^1$Department of Mathematics, University of Maragheh, Maragheh, Iran.}

\email{rahimi@maragheh.ac.ir}
\address{ $^{2}$Department of Mathematics, Payame Noor University, Iran.}
\email{najafzadeh1234@yahoo.ie}
\address{  $^{3}$Department of Mathematics, Payame Noor University, Iran.}
\email{nouri@yahoo.com}


\begin{abstract}
\kf\  were recently introduced by L. G\v{a}vruta in \hs\ to study atomic systems with respect to bounded linear operator. Also controlled frames have been recently introduced by P. Balazs in \hs\ to improve the numerical efficiency of interactive algorithms for inverting the frame operator. In this manuscript, we will define the concept of the controlled \kf\ and  will show that controlled \kf\  are equivalent to \kf\  and so the controlled operator $C$ can be used as preconditions in applications.



\end{abstract}
\maketitle
\section{Introduction}
Frames in \hs\ were first proposed by Duffin and Schaeffer to deal with nonharmonic Fourier series in 1952 \cite{Duff} and widely studied from 1986 since the great work by Daubechies at al.\cite{Daub}.
Now frames play an important role not only in the theoretics but also in many kinds of applications and have been widely applied in signal processing \cite{Ferr}, sampling \cite{Eldar1,Eldar2}, coding and communications \cite{strh}, filter bank theory \cite{Bolcskei}, system modeling \cite{Duday} and so on.
For special applications many other types of frames were proposed, such as the fusion frames \cite{Casazza1,Casazza2} to deal with hierarchical data processing, $g$-frames \cite{sun} by Sun to deal with all existing frames as united object, oblique dual frames \cite{Eldar1} by Elder to deal with sampling reconstructions, and so on. \par
The notion of
\kf\    were recently introduced by L. G\v{a}vruta to study the atomic systems with respect to a bounded linear operator $K$ in \hs. From \cite{12}, we know that \kf\   are more general than ordinary frames in sense that the lower frame bound only holds for the elements in the range of the  $K$, where $K$ is a bounded linear operator in a separable Hilbert Space $H$.\par
Controlled frames have been introduced recently to improve the numerical efficiency of interactive algorithms for inverting the frame operator on abstract \hs\ \cite{Balazs}, however they are used earlier in \cite{Bogdanova} for spherical wavelets. This concept generalized for fusion frames in \cite{khos} and $g$-frames in \cite{Rahimi}.\par
In this manuscript, the concept of controlled $K$-frame will be defined and it will be shown that any controlled $K$-frame is equivalent to a $K$-frame and the role of controller operators are like the role of preconditions matrices or operators in linear algebra.\par
This paper is organized as follows. In section 2, we fix the notations of this paper, summarize known and prove some new results needed for the rest of the paper. In section 3, we define the concept of controlled \kf\  and we show that a controlled $K$-frame is equivalent to a $K$-frame.\par
Throughout this manuscript, $H$ is a separable \hs\ , $K$ is a bounded linear operator on $H$, $B(H)$ the family of all linear bounded operators on $H$ and $GL(H)$ the set of all bounded invertible operators on $H$ with bounded inverse.
\section{Preliminaries and Notations}
Now we state some notations and theorems which are used in the present paper.

A bounded operator $T\in B(H)$ is called positive (respectively, non-negative), if $\langle Tf,f\rangle>0$ for all $f\ne0$ (respectively, $\langle Tf,f\rangle\ge0$ for all $f$).
Every non-negative operator is clearly self-adjoint.
If $A\in B(H)$  is non-negative, then there exists a unique non-negative operator $B$ such that $B^2=A$. Furthermore $B$ commutes with every operator that commutes with $A$.
This will be denoted by $B=A^{\frac{1}{2}}$. Let $B^+(H)$ be the set of positive operators on $H$. For self-adjoint operators $T_1$ and $T_2$, the notation $T_1\leq T_2$ or $T_2-T_1\geq 0$ means
$$ \langle T_1 f,f\rangle\leq \langle T_2f,f\rangle \quad \forall f\in H.  $$
The following result is needed in the sequel, but straightforward to prove:
\begin{prop}\label{prp:equs}
Let $T:\ H\to H$ be a linear operator. Then the following condition are equivalent:
\begin{enumerate}
\item There exist $m>0$ and $M<\infty$, such that $mI\le T\le MI$;
\item $T$ is positive and there exist $m>0$ and $M<\infty$, such that $m\|f\|^2\le \|T^{\frac{1}{2}}f\|^2\le M\|f\|^2$ for all $f\in H$;
\item $T$ is positive and $T^{\frac{1}{2}}\in GL(H)$;
\item There exists a self-adjoint operator  $A\in GL(H)$
, such that \[A^2=T;\]
\item $T\in GL^+(H)$;
\item There exist constants $m>0$ and $M<\infty$ and operator\\ $C\in GL^+(H)$, such that $m'C\le T\le M'C$;
\item For every $C\in GL^+(H)$, there exist constants $m>0$ and\\ $M<\infty$, such that $m'C\le T\le M'C$.
\end{enumerate}
\end{prop}
\begin{defn}
Given $T\in GL^+(H)$, any two constants $m_T,\ M_T$ such that
\[m_TI\le T\le M_TI\]
are called lower and upper bounds of $T$, respectively. If $m_T$ is maximal, resp. if $M_T$ is minimal, we call them the optimal bounds and we denote them by $m_T^{(opt)}$, $M_T^{(opt)}$.
The upper and lower bounds are clearly not unique.\\
The following results are easily proved using Proposition \ref{prp:equs}:
\end{defn}
\begin{cor}
Let $T\in GL^+(H)$. then
\begin{enumerate}
\item $\|T\|=M_T^{(opt)}$
\item $\sigma(T)\subseteq[m_T,M_T]$, for any lower, resp. upper bounds.
\end{enumerate}
\end{cor}
\begin{cor}
For $T\in GL^+(H)$, the numbers $m_{T^{-1}}=M_T^{-1}$ and $M_{T^{-1}}=m_T^{-1}$ are bounds for $T^{-1}$. In particular $\|T^{-1}\|=\dfrac{1}{m_T^{(opt)}}$
\end{cor}
Clearly if  there exists $0<m\le M<\infty$ such that
\[m\le T\le M\] then
for $T^{-1}$ we have
\[M^{-1}\le T^{-1}\le m^{-1}.\]
It is well-known that not all bounded operators $U$ on a Hilbert space $H$
are invertible: an operator $U$ needs to be injective and surjective in order
to be invertible. For doing this, one can use right-inverse operator.  The following lemma shows that if an operator $U$ has closed range,
there exists a "right-inverse operator" $U^\dagger$ in the following sense:

\begin{lem}\label{111}\cite{ole}
Let $H$ and $K$ be Hilbert spaces and suppose that $U : K\to H$
is a bounded operator with closed range $R_U$. Then there exists a bounded
operator $U^\dagger :H\to K$ for which $$ UU^\dagger x=x\quad\forall x\in R_U.$$
\end{lem}
The operator $U^\dagger$ in the  Lemma \ref{111} is called the
pseudo-inverse of $U$. In the literature, one will often see the pseudo-inverse
of an operator $U$ with closed range defined as the unique operator $U^\dagger$
satisfying that
$$
N_{U^\dagger}=R_U^\perp,\quad R_{U^\dagger}=N_U^\perp,\quad UU^\dagger x=x\quad\forall x\in R_U.
$$
\subsection{K-Frames}
Let us first introduce the concepts of $K$-frames, the atomic system of $K$ and the related operators.
\begin{defn}\label{def:kframe}
A sequence $\{f_n\}_{n=1}^\infty\subset H$  is called a $K$-frame for $H$, if there exist constants $A, B>0$ such that
\begin{equation}\label{eq:kframe}
A\|K^*f\|^2\le \sum_{n=1}^{\infty}|\langle f,f_n\rangle|^2\le B\|f\|^2,~~\forall f\in H.
\end{equation}
we call $A$ and $ B$  lower and upper frame bound for $K$-frame $\{f_n\}_{n=1}^\infty\subset H$, respectively if only the right inequality of \ref{def:kframe} holds, $\{f_n\}_{n=1}^\infty\subset H$ is called a $K$-Bessel sequence.
\end{defn}
\begin{rem}
If $K=I$, then $K$-frames are just the ordinary frames.
\end{rem}
\begin{rem}
In the following we will assume that $R(K)$ is closed, since this can assure that the pseudo-inverse $K^\dagger$ of $K$ exists.
\end{rem}
\begin{defn}\label{def:atom}\cite{Gavruta}
A sequence $\{f_n\}_{n=1}^\infty\subset H$  is called an atomic system for $K$, if the following conditions are satisfied:
\begin{enumerate}
\item $\{f_n\}_{n=1}^\infty$ is a Bessel sequence.
\item For any $x\in H$, there exists $a_x=\{a_n\}\in l^2$ such that
\[kx=\sum_{n=1}^{\infty}a_nf_n\]
where $\|a_x\|_{l^2}\le C\|x\|$, $C$ is positive constant.
\end{enumerate}
\end{defn}
Suppose that $\{f_n\}_{n=1}^\infty$ is a $K$-frame for $H$. Obviously it is a Bessel sequence, so we can define the following operator
\[T:l^2\to H,\quad Ta=\sum_{n=1}^{\infty}a_nf_n,\quad a=\{a_n\}\in l^2,\]
then we have
\[T^*:H\to l^2\]
\[T^*f=\{\langle f,f_n\rangle\}_{n=1}^\infty.\]
Let $S=TT^*$, we obtain
\[Sf=\sum_{n=1}^{\infty}\langle f,f_n\rangle f_n  \quad \forall f\in H \]
we call $T,\ T^*$ and $\ S$ the synthesis operator, analysis operator and frame operator for $K$-frame $\{f_n\}_{n=1}^\infty$, respectively.

\begin{thm}
Let $\{f_n\}_{n=1}^\infty$ be a Bessel sequence in $H$. Then $\{f_n\}_{n=1}^\infty$ is a $K$-frame for $H$, if and only if there exists $A>0$ such that \[S\ge AKK^*,\] where $S$ is the frame operator for $\{f_n\}_{n=1}^\infty$.
\end{thm}
\begin{proof} The sequence
$\{f_n\}_{n=1}^\infty$ is a $K$-frame for $H$ with frame bounds $A, B$ and frame operator $S$, if and only if
\begin{equation}\label{eq:th25}
A\|K^*f\|^2\le\sum_{K=1}^{\infty}|\langle f,f_n\rangle |^2=\langle Sf,f\rangle\le B\|f\|^2\ ,~\forall f\in H,
\end{equation}
that is ,
\[\langle AKK^*f,f\rangle\le\langle Sf,f\rangle\le \langle Bf,f\rangle\ ,~\forall f\in H.\]
so the conclusion holds.
\end{proof}
\begin{rem}
Frame operator of a \kf\ is not invertible on $H$ in general, but we can show that it is invertible on the subspace $R(K)\subset H$. In fact, since $R(K)$ is closed, there exists a pseudo-inverse $K^\dagger$ of $K$, such that
$KK^\dagger f=f\ ,~\forall f\in R(K)\ ,~namely\  KK^\dagger|_{R(K)}=I_{R(K)}$, so we have $I_{R(K)}^*=(K^\dagger|_{R(K)})^*K^*$. Hence for any $f\in R(K)$, we obtain
\[\|f\|=\|(K^\dagger|_{R(K)})^*K^*f\|\le\|K^\dagger\|.\|K^*f\|,\]
that is, $\|K^*f\|^2\ge \|K^\dagger\|^{-2}\|f\|^2$. Combined with (\ref{eq:th25}) we have
\begin{equation}
\langle Sf,f\rangle\ge A\|K^*f\|^2\ge A\|K^\dagger\|^{-2}\|f\|^2\ ,~\forall f\in R(K).
\end{equation}
So, from the definition of $K$-frame we have
\begin{equation}
A\|K^\dagger\|^{-2}\|f\|\le\|Sf\|\le B\|f\|\ ,~\forall f\in R(K),
\end{equation}
wich implies that $S:\ R(K)\to S(R(K))$ is a homeomorphism, furthermore, we have
\[B^{-1}\|f\|\le \|S^{-1}f\|\le A^{-1}\|K^\dagger\|^2\|f\|\ ,~\forall f\in S(R(K)).\]
\end{rem}
\begin{prop}
Let $\{f_n\}_{n=1}^\infty\subset H$. Then $\{f_n\}_{n=1}^\infty$ is a $K$-frame for $H$, if and only if there exist a Bessel sequence $\{g_n\}_{n=1}^\infty\subset H$ such that
\begin{equation}\label{ex:prp27}
Kf=\sum_{n=1}^{\infty}\langle f,g_n\rangle f_n~,~\forall f\in H.
\end{equation}\end{prop}
One may wonder whether the position of the two Bessel sequences $\{f_n\}_{n=1}^\infty$ and $\{g_n\}_{n=1}^\infty$ in (\ref{ex:prp27}) are interchangeable? In fact, the answer us negative.

\begin{exa}
Suppose that $H=\CC^3$, $\{g_n\}_{n=1}^3=\{e_1,e_2,e_3\}$, where
$e_1=\begin{pmatrix}1\\ 0\\0\end{pmatrix}$,
$e_2=\begin{pmatrix}0\\ 1\\0\end{pmatrix}$,
$e_3=\begin{pmatrix}0\\ 0\\1\end{pmatrix}$.
Now define $K\in B(H)$ as follows
\[K:\ H\to H~,~Ke_1=e_1~,~ke_2=e_1~,~Ke_3=e_2.\]
Obviously, $\{g_n\}_{n=1}^3$ is an ordinary frame for $H$. By Proposition \ref{prp: fgh}  we know that $$\{f_n\}_{n=1}^3=\{Kg_n\}_{n=1}^3$$ is a $K$-frame for $H$.
Since $\{g_n\}_{n=1}^3=\{e_1,e_2,e_3\}$, so for any $f\in H$, we have $Kf=\sum_{n=1}^{3}\langle f,g_n\rangle g_n$, it follows that
\[Kf=\sum_{n=1}^{3}\langle f,g_n\rangle Kg_n=\sum_{n=1}^{3}\langle f,g_n\rangle f_n~,~\forall f\in H.\]
But $Kf\ne\sum_{n=1}^{3}\langle f,f_n\rangle g_n~,~\forall f\in H$. In fact, if we take $f=e_3$ , then we have
\[Ke_3=e_2\ne\sum_{n=1}^{3}\langle e_3,f_n\rangle g_n=\langle e_3,e_1\rangle e_1+\langle e_3,e_1\rangle e_2+\langle e_3,e_2\rangle e_3=0.\]
Through $\{f_n\}_{n=1}^\infty$ and $\{g_n\}_{n=1}^\infty$ in (\ref{ex:prp27}) are not interchangeable in general, we can show that there exists another type of dual such that $\{f_n\}_{n=1}^\infty$ and a sequence derived by $\{g_n\}_{n=1}^\infty$ are interchangeable in the subspace $R(K)$.
\end{exa}
\begin{prop}\label{prp: fgh}
Suppose that $\{f_n\}_{n=1}^\infty$ and $\{g_n\}_{n=1}^\infty$ are as in (\ref{ex:prp27}). then there exists a sequence
$\{h_n\}_{n=1}^\infty=\{(K^\dagger|_{R(K)})^*g_n\}_{n=1}^\infty$ derived by $\{g_n\}_{n=1}^\infty$ such that
\[f=\sum_{n=1}^{\infty}\langle f,h_n\rangle f_n~,~\forall f\in R(K).\]
Moreover, $\{h_n\}_{n=1}^\infty$ and $\{f_n\}_{n=1}^\infty$ are interchangeable for any $f\in R(K)$.
\end{prop}
\begin{prop}
If $\{f_n\}_{n=1}^\infty$ is an ordinary frame for $H$, then $\{Kf_n\}_{n=1}^\infty$ is a $K$-frame for $H$.
\end{prop}
\begin{prop}
If $\{e_n\}_{n=1}^\infty$ is an orthogonal basis for $H$, then $\{Ke_n\}_{n=1}^\infty$ is a $K$-frame for $H$.
\end{prop}
\begin{prop}
If $T\in B(H)$ and $\{f_n\}_{n=1}^\infty$ is $K$-frame for $H$, then $\{Tf_n\}_{n=1}^\infty$ is a T$K$-frame for $H$.
\end{prop}
\begin{proof}
It follows directly by the definition of $K$-frame.
\end{proof}

\section{Controlled $K$-frames}
Controlled frames for spherical wavelets were introduced in \cite{Bogdanova} to get a numerically
more efficient approximation algorithm and the related theory for general frames was
developed in \cite{Balazs}. For getting a numerical solution of a linear system of equations $Ax = b$,
we can solve the system of equations $PAx = Pb$, where $P$ is a suitable preconditioning
matrix to get a better iterative algorithm, which was the main motivation for introducing controlled frames in \cite{Bogdanova}. Controlled frames extended to $g$-frames in \cite{Rahimi} and for fusion frames in \cite{khos}.
In this section, the concept of controlled frames and controlled Bessel sequences will be extended to $K$-frames and we will show that controlled $K$-frames are equivalent $K$-frames.

\begin{defn}
Let $C\in GL^+(H)$ ($C>0$) and let $CK=KC$. The family $\{f_n\}_{n=1}^\infty$ will called a $C$-controlled \kf\ for $H$, if $\{f_n\}_{n=1}^\infty$ is a $K$-Bessel sequence and there exist constants $A>0$ and $B<\infty$ such that
\[A\|C^{\frac{1}{2}}K^*f\|^2\le\sum_{n=1}^{\infty}\langle f,f_n\rangle\langle f,Cf_n\rangle\le B\|f\|^2\ ,~\forall f\in H.\]

$A$ and $B$ will be called $C$-controlled $K$-frame bounds. If $C=I$, we call $\{f_n\}_{n=1}^\infty$ a \kf\ for $H$ with bounds $A$ and $B$.

If the second part of the above inequality holds, it will be called $C$-controlled $K$-Bessel sequence with bound $B$.
\end{defn}

the proof of the following lemmas is straightforward.

\begin{lem}
Let $C>0$ and  $C\in GL^+(H)$. The $K$-Bessel sequence $\{f_n\}_{n=1}^\infty$ is $C$-controlled $K$-Bessel sequence if and only if there exists constant $B<\infty$ such that
\[\sum_{n=1}^{\infty}\langle f,f_n\rangle\langle f,Cf_n\rangle\le B\|f\|^2\ ,~\forall f\in H.\]
\end{lem}

\begin{lem}
Let $C\in GL^+(H)$. A sequence $\{f_n\}_{n=1}^\infty\in H$ is a $C$-controlled Bessel sequence for $H$ if and only if the operator
\[L_C:H\to H\ ,~L_Cf=\sum_{n=1}^{\infty}\langle f,f_n\rangle Cf_n,\quad f\in H\]
is well defined and there exists constant $B<\infty$ such that
\[\sum_{n=1}^{\infty}\langle f,f_n\rangle\langle f,Cf_n\rangle\le B\|f\|^2\ ,~\forall f\in H.\]
\end{lem}

\begin{rem}
The operator $L_C:H\to H\ ,~L_Cf=\sum_{n=1}^{\infty}\langle f,f_n\rangle Cf_n, f\in H$ is called the $C$-controlled Bessel sequence operator, also $L_Cf=CSf$.
\end{rem}

\begin{lem}
Let $\{f_n\}_{n=1}^\infty$ be a C-controlled \kf\ in $H$, for $C\in GL^+(H)$. Then
\[AI\|C^{\frac{1}{2}}K^\dagger\|^2\le L_C\le BI\]
\end{lem}
\begin{proof}
Suppose that  $\{f_n\}_{n=1}^\infty$ is a $C$-controlled $K$-frame with bounds $A$ and $B$. Then
\[A\|C^{\frac{1}{2}}K^*f\|^2\le\sum_{n=1}^{\infty}\langle f,f_n\rangle\langle f,Cf_n\rangle\le B\|f\|^2\ ,~\forall f\in H.\]
For $f \in H$
\[A\|C^{\frac{1}{2}}K^*f\|^2\le\langle f,L_Cf\rangle\le B\|f\|^2\]
that is
\[A\|C^{\frac{1}{2}}K^*\|^2\le L_C\le BI.\]
\end{proof}
The following proposition shows that for evaluation a family $\{f_n\}_{n=1}^\infty\subset H$ to be a controlled $K$-frame it is suffices to check just a simple operator inequality.
\begin{prop}
Let $\{f_n\}_{n=1}^\infty$ be a Bessel sequence in $H$ and $C\in GL^+(H)$. Then  $\{f_n\}_{n=1}^\infty$ is a $C$-controlled $K$-frame for $H$, if and only if there exists $A>0$ such that $CS\ge CAKK^*.$
\end{prop}
\begin{proof} The sequence
$\{f_n\}_{n=1}^\infty$ is a controlled \kf\ for $H$ with frame bounds $A,B$ and frame operator $S$, if and only if
\[A\|C^{\frac{1}{2}}K^*f\|^2\le\sum_{n=1}^{\infty}\langle f,f_n\rangle\langle f,Cf_n\rangle\le B\|f\|^2\ ,~\forall f\in H.\]
That is,
\[\langle CAKK^*f,f\rangle\le\langle CSf,f\rangle\le\langle Bf,f\rangle.\]
So the conclusion holds.
\end{proof}

\begin{prop}
Let $\{f_n\}_{n=1}^\infty$ be a $C$-controlled $K$-frame and $C\in GL^+(H)$. Then  $\{f_n\}_{n=1}^\infty$ is a \kf\ for $H$.
\end{prop}
\begin{proof}
Suppose that $\{f_n\}_{n=1}^\infty$ is a controlled \kf\  with bounds $A$ and $B$. Then
\[A\|K^*f\|^2=A\|C^{-\frac{1}{2}}C^{\frac{1}{2}}K^*f\|^2\le
A\|C^{\frac{1}{2}}\|^2\|C^{-\frac{1}{2}}K^*f\|^2\]
\[\le \|C^{\frac{1}{2}}\|^2\sum_{n=1}^{\infty}\langle f,f_n\rangle\langle f,C^0f_n\rangle=
\|C^{\frac{1}{2}}\|^2\sum_{n=1}^{\infty}|\langle f,f_n\rangle|^2\]
Hence
\[A\|C^{\frac{1}{2}}\|^{-2}\|K^*f\|^2\le \sum_{n=1}^{\infty}|\langle f,f_n\rangle|^2\]
On the other hand for every $f\in H$,
\begin{eqnarray*}
\sum_{n=1}^{\infty}|\langle f,f_n\rangle|^2 & = & \langle f,Sf\rangle=\langle< f,C^{-1}CSf\rangle\\
 & = &  \langle(C^{-1}CS)^{\frac{1}{2}}f,(C^{-1}CS)^{\frac{1}{2}}f\rangle\\
 & = & \|(C^{-1}CS)^{\frac{1}{2}}f\|^2\\
 & \le & \|(C^{-\frac{1}{2}}\|^2\|(CS)^{\frac{1}{2}}f\|^2\\
 & = & \|(C^{-\frac{1}{2}}\|^2\langle f,CSf\rangle\\
 & \le & \|(C^{-\frac{1}{2}}\|^2B\|f\|^2
\end{eqnarray*}
\end{proof}

These inequalities yields that $\{f_n\}_{n=1}^\infty$ is a $K$-frame  with bounds $A\|C^{\frac{1}{2}}\|^{-2}$ and $B\|C^{-\frac{1}{2}}\|^2$.

\begin{prop}
Let $C\in GL^+(H)$  and $KC=CK$, Let $\{f_n\}_{n=1}^\infty$ be \kf\ for $H$, then $\{f_n\}_{n=1}^\infty$ is a $C$-controlled frame for $H$.
\end{prop}
\begin{proof}
Suppose that $\{f_n\}_{n=1}^\infty$ be a \kf\ with bounds $A'$ and $B'$. Then for all $f\in H$
\[A'\|K^*f\|^2\le\sum_{n=1}^{\infty}|\langle f,f_n\rangle|^2\le B'\|f\|^2.\]
\begin{eqnarray*}
A'\|C^{\frac{1}{2}}K^*f\|^2 = A'\|K^*C^{\frac{1}{2}}f\|^2 & \le &
\sum_{n=1}^{\infty}\langle C^{\frac{1}{2}}f,f_n\rangle\langle C^{\frac{1}{2}}f,f_n\rangle\\
 & = & \langle C^{\frac{1}{2}}f,\sum_{n=1}^{\infty}\langle f_n,C^{\frac{1}{2}}f\rangle f_n\rangle\\
 & = & \langle C^{\frac{1}{2}}f, C^{\frac{1}{2}}Sf\rangle = \langle f,CSf\rangle
 \end{eqnarray*}
Hence $A'\|C^{\frac{1}{2}}K^*f\|^2\le\langle f,CSf\rangle$ for every $f\in H $. On the other hand for every  $f\in H$,
\[|\langle f,CSf\rangle|^2=|\langle C^*f,Sf\rangle|^2=|\langle Cf,Sf\rangle|^2\le\|Cf\|^2\|Sf\|^2\le\|C\|^2\|f\|^2B\|f\|^2.\]
Hence \[A'\|C^{\frac{1}{2}}K^*f\|^2\le\langle f,CSf\rangle\le B'\|C\|\|f\|^2.\] \\
Therefore $\{f_n\}_{n=1}^\infty$ is a $C$-controlled $K$-frame with bounds $A'$ and $B'\|C\|$.
\end{proof}

\begin{lem}
Let $\{f_n\}_{n=1}^\infty$ be a C-controlled \kf\ in $H$ for $C\in GL^+(H)$ and let $C$ is positive. Then
\[\sum_{n=1}^{\infty}\langle f_n,f\rangle Cf_n=\sum_{n=1}^{\infty}\langle Cf_n,f\rangle f_n\ ,~\forall f\in H.\]
\end{lem}
\begin{proof}
Suppose  $\{f_n\}_{n=1}^\infty$ be a controlled frame for $H$ with bounds $A$ and $B$. Then by Proposition \ref{prp:equs}, we know that $L_C\in GL(H)$. Let $\widetilde{L}=C^{-1}L_C$. Clearly $\widetilde{L}\in GL(H)$ and
\[\widetilde{L}f=C^{-1}\sum_{n=1}^{\infty}\langle f_n,f\rangle Cf_n=\sum_{n=1}^{\infty}\langle f_n,f\rangle f_n=Lf.\]
Therefore $L$ is everywhere defined and $L\in GL(H)$. Thus by definition $L_C$ is positive, therefore self-adjoint. So
\[L_C=CL=L_C^*=L^*C^*=LC^*.\]
\end{proof}


\end{document}